\documentclass[onefignum,onetabnum]{amsart}

\usepackage[T1]{fontenc}
\usepackage[english]{babel}
\usepackage{hyperref}
\usepackage{geometry}
\usepackage{pdflscape}
\usepackage{amsfonts,amssymb}
\usepackage{graphicx}
\usepackage{tikz}
\usepackage{pgfplots}
\usepackage{caption}
\usepackage[labelfont=rm]{subcaption}
\usepackage{xspace}
\usepackage{booktabs,tabularx}
\usepackage{dcolumn}
\usepackage{listings}
\usepackage{marvosym}
\usepackage{microtype}
\usepackage{paralist}
\usepackage{csquotes}
\usepackage{amsopn}

\definecolor{links}{rgb}{.2,.1,.5}
\definecolor{cites}{rgb}{.5,.1,.2}
\hypersetup{colorlinks,linkcolor=links,citecolor=cites}

\newcolumntype{d}[1]{D{.}{.}{#1} }

\newcommand\cH{{\mathcal H}}

\newcommand\NN{{\mathbb N}}

\newcommand\RR{{\mathbb R}}
\newcommand\TTmin{{\mathbb T}_{\min}}
\newcommand\TTmax{{\mathbb T}_{\max}}
\newcommand\ZZ{{\mathbb Z}}

\newcommand{\1}{\mathbf 1}

\newcommand\SetOf[2]{\left\{\left.#1\vphantom{#2}\ \right|\ #2\vphantom{#1}\right\}}
\newcommand\smallSetOf[2]{\{{#1}\,|\,{#2}\}}

\DeclareMathOperator{\supp}{supp}

\newcommand\bigO{\mathcal{O}}

\newcommand\minunit{\mathcal{E}_{\min}}
\newcommand\maxunit{\mathcal{E}_{\max}}
\newcommand\monomial[1]{{\sf M}(#1)}
\newcommand\closedmonomial[1]{\overline{\sf M}(#1)}
\newcommand\complementarymonomial[1]{\rotatebox[origin=c]{180}{\sf M}(#1)}
\newcommand\closedcomplementarymonomial[1]{\overline{\rotatebox[origin=c]{180}{\sf M}}(#1)}
\newcommand\hahnseries[2]{{#1}\{\hskip-.25em\{{#2}^\RR\}\hskip-.25em\}}
\DeclareMathOperator\val{ord}
\newcommand\bm{\mathbf}
\DeclareMathOperator\topint{int}

\newcommand\tmi{min} %%{$\min$} 
\newcommand\tma{max} %%{$\max$}

\tikzset{Apex1/.style={draw=black, fill=green!50, circle, inner sep = \nodesize}}
\tikzset{Apex2/.style={draw=black, fill=red!50, circle, inner sep = \nodesize}}
\tikzset{Apex3/.style={fill=black, circle, inner sep = \nodesize}}
\tikzset{InfGen/.style={fill=black, circle, inner sep = \nodesize}}
\tikzset{Border/.style={black,very thick}}
\tikzset{ConeLabel/.style={}}
\tikzset{SearchRegion/.style={red!20}}
\tikzset{DominatedSet/.style={green!20}}
\tikzset{ExtGen/.style={draw=black, circle, inner sep = 2.5pt, fill=black!40}}

\newcommand\inflen{6}

\newcommand{\boundingbox}[5]{
  \coordinate (corner1) at (#1, #2);
  \coordinate (corner2) at (#3, #2);
  \coordinate (corner3) at (#3, #4);
  \coordinate (corner4) at (#1, #4);

  \draw[dashed] (corner1) -- (corner2) -- (corner3) -- (corner4) -- cycle;

  \node [label=below left:{$#5$}] at (corner3) {};
}

\newcommand{\fillbox}[4]{
  \coordinate (corner1) at (#1, #2);
  \coordinate (corner2) at (#3, #2);
  \coordinate (corner3) at (#3, #4);
  \coordinate (corner4) at (#1, #4);

  \fill[blue!10] (corner1) -- (corner2) -- (corner3) -- (corner4) -- cycle;
}

\newcommand\nodesize{1.7pt}

\usepackage{algorithmicx}
\usepackage{algorithm}
\usepackage{algpseudocode}

\MakeRobust{\Call}

\algdef{SE}[DOWHILE]{Do}{DoWhile}{\algorithmicdo}[1]{\algorithmicwhile\ #1}
\algrenewcommand{\algorithmiccomment}[1]{\hfill $\rhd$ \emph{#1}}%{\hskip1em $\rhd$ \emph{#1}}
\algrenewcommand{\algorithmicrequire}{\textbf{Input:}}
\algrenewcommand{\algorithmicensure}{\textbf{Output:}}
\algnewcommand{\Or}{\textbf{or}}
\algnewcommand{\And}{\textbf{and}}
\algnewcommand{\Not}{\textbf{not}\,}

\theoremstyle{plain}
\newtheorem{theorem}{Theorem}
\newtheorem{lemma}[theorem]{Lemma}
\newtheorem{proposition}[theorem]{Proposition}
\newtheorem{corollary}[theorem]{Corollary}
\theoremstyle{definition}
\newtheorem{remark}[theorem]{Remark}
\newtheorem{example}[theorem]{Example}
\newtheorem{question}[theorem]{Question}

\title[Monomial tropical cones for multicriteria optimization]{Monomial tropical cones for \\ multicriteria optimization}

\author{Michael Joswig \and Georg Loho}

\thanks{Research by M. Joswig is partially supported by Einstein Stiftung Berlin and Deutsche Forschungsgemeinschaft (EXC 2046: \enquote{MATH$^+$}, SFB-TRR 109: \enquote{Discretization in Geometry and Dynamics}, SFB-TRR 195: \enquote{Symbolic Tools in Mathematics and their Application}, and GRK 2434: \enquote{Facets of Complexity}).  Additional support by Institut Mittag-Leffler within the program \enquote{Tropical Geometry, Amoebas and Polytopes} is gratefully acknowledged.}

\address{Technische Universit{\"a}t Berlin, Chair for Discrete Mathematics/Geometry, and MPI MiS Leipzig }
\email{joswig@math.tu-berlin.de}
\address{London School of Economics and Political Science, UK }
\email{g.loho@lse.ac.uk}



\begin{document}

\begin{abstract}
  We present an algorithm to compute all $n$ nondominated points of a multicriteria discrete optimization problem with $d$ objectives using at most $\bigO(n^{\lfloor d/2 \rfloor})$ scalarizations.
  The method is similar to algorithms by Przybylski et al.\ (2010) and by Klamroth et al.\ (2015) with the same complexity.
  As a difference, our method employs a tropical convex hull computation, and it exploits a particular kind of duality which is special for the tropical cones arising.
  This duality can be seen as a generalization of the Alexander duality of monomial ideals.
\end{abstract}

\maketitle
 
\section{Introduction}
\noindent
In practical applications of optimization it may occur that there are competing choices for objective functions.
Classical examples include multiple knapsack problems, where one knapsack is to be filled with various items, but the value of each item may depend on individual preferences of various people to decide what to take into the knapsack.
While, in general, it is beyond mathematics to resolve the conflicts of interest arising, it is a relevant task for optimization to exhibit the trade-offs and especially to find those feasible solutions which are locally optimal.
In multicriteria optimization these local optima are known as \emph{Pareto optima}.
Their images in the outcome space are the \emph{nondominated points}.
The main purpose of this paper is to interpret a known technique for computing all nondominated points for a given discrete multicriteria optimization problem in the context of tropical geometry.
Our key observation is that the nondominated points arise as the extremal generators of a special kind of tropical cone.

Tropical geometry is a mathematical field which connects computations in the $(\min,+)$-semiring with other disciplines, including algebraic geometry, commutative algebra, graph theory, statistics, polyhedral geometry, and optimization; cf.~\cite{Tropical+Book} for a general introduction to the subject.
The branch which is most relevant for our purposes is known as tropical linear algebra; cf.~\cite{Butkovic10}.
Today this is often also called \emph{tropical convexity} to stress the key geometric features of that theory which, as in our case, often lead to natural algorithms.
It is a fundamental fact that the tropical cones, which are precisely the $(\min,+)$-semimodules, arise as projections of ordinary convex cones defined over the ordered field of formal Puiseux series with real coefficients \cite{DevelinYu:2007}.
That projection is induced by the valuation map which sends a Puiseux series to its lowest exponent.
Since the real Puiseux series form a real closed field it follows that polyhedral cones, convex polyhedra and polytopes, linear programming etc.\ work precisely as over the real numbers.
In this way, tropical cones and polyhedra inherit many properties and algorithms from their classical counterparts.
The benefit is substantial: This approach shows how the algorithms for determining the set of nondominated points of a discrete multicriteria optimization problem obtained in~\cite{PrzybylskiGandibleuxEhrgott:2010, KLV2015, DaechertEtAl} can be considered as (dual) tropical convex hull computations.
Interestingly, the tropical cones arising in this setting are quite special.
In fact, they can be viewed as generalizations of the monomial ideals arising in commutative algebra; e.g., see \cite{MillerSturmfels:2005,HerzogHibi:2011} and Section~\ref{sec:concluding} below.
Hence we suggest \emph{monomial tropical cone} as a name.

Before we will get to describe our contribution in greater detail we will now formally define our objects of study.
A \emph{multicriteria optimization problem} is of the form
\begin{equation}\label{eq:multiopt}
  \begin{array}{ll}
    \min & f(x) = \bigl( f_1(x), \ldots, f_d(x) \bigr) \\
    \mbox{subject to } & x \in X \enspace .
  \end{array}
\end{equation}
Here $X$ is the \emph{feasible set}.
It is a subset of the \emph{decision space}, which may be any set.
The \emph{objective functions} $f_i$ have the feasible set as their common domain, and they take real values.
We will mainly deal with the image $Z = f(X)$ of the feasible set, the \emph{outcome space}, which is a subset of $\RR^d$.
A point $z \in Z$ is \emph{nondominated} if there is no point $w \in Z$ such that $w_i \leq z_i$ for all $i \in [d] := \{1,2,\dots,d\}$ and $w_{\ell} < z_{\ell}$ for at least one $\ell \in [d]$. 
The set of all nondominated points in $Z$ is the \emph{nondominated set}.
Each nondominated point can be obtained by determining an optimal solution of a scalarization of the multiobjective problem \cite{Ehrgott:2005}.
The latter is an optimization problem derived from \eqref{eq:multiopt} by suitably restricting the feasible set and optimizing with respect to just one objective function derived from $f$.
There are general methods known to determine all nondominated points by successively choosing appropriate scalarizations.
Typically these scalarizations are considered computationally expensive, whence the complexity of a multicriteria optimization problem is measured in the number of scalarizations required.
If $d$ is fixed, the asymptotically tight upper bound is $\bigO(n^{\lfloor d/2\rfloor})$, where $n$ is the number of nondominated points. 
This follows from work of Kaplan et al.\ \cite{KaplanEtAl:2008} on colored orthogonal range counting.
D{\"a}chert et al.\ \cite{DaechertEtAl} presented an enumeration strategy via scalarizations with respect to `boxes' and `local upper bounds', which is asymptotically optimal.
This builds on earlier work using similar decompositions of the search space in~\cite{KLV2015,PrzybylskiGandibleuxEhrgott:2010}.
Our algorithm can be viewed as a variation of their idea and requires the same number of scalarizations.
The essential new contribution is the observation that the nondominated set can be interpreted as the extremal generators of a certain kind of tropical cone.
This allows us to use an adaptation of the tropical double description method \cite{DoubleDescription:2010} to deduce an enumeration scheme which results in the same number of subproblems and is also asymptotically worst case optimal.
In this way the known upper bound can also be derived from the tropical upper bound theorem of Allamigeon, Gaubert and Katz \cite{AGK:ExtremePoints2011}.
The ordinary double description method, also known as Fourier--Motzkin elimination, is a standard algorithm for computing (dual) ordinary convex hulls \cite{FukudaProdon:1996}.
The dual convex hull problem asks to convert an exterior description of an ordinary convex polyhedron (in terms of linear inequalities) into an interior description (in terms of generating points and rays).
It can be seen as a parameterized linear optimization problem where the feasible region is fixed and the linear objective function is allowed to vary arbitrarily.
Our results show that all discrete multicriteria optimization problems exhibit the same geometric structure. 
It turns out that the monomial tropical cones arising in multicriteria optimization already made an appearance as the `$i$th polar cones' in \cite{AllamigeonGaubertKatz:2011}.
Yet, apparently, they have not been studied in full detail before.

In the remainder of this introduction we will give an outline of the present article.
Section~\ref{sec:tropical} starts out with the basic notions from tropical convexity.
Our first main result, Theorem~\ref{thm:complementary-cones}, states that monomial tropical cones always come in pairs, one with respect to $\max$ and the other with respect to $\min$ as the tropical addition.
This can be seen as a generalization of Alexander duality of monomial ideals \cite[\S5.2]{MillerSturmfels:2005}.
The subsequent Section~\ref{sec:upper-bound} is devoted to deriving an upper bound for the number of generators of the dual monomial tropical cone in terms of the number of generators of the primal tropical cone.
This follows from the tropical upper bound theorem of Allamigeon, Gaubert and Katz \cite{AllamigeonGaubertKatz:2011}.
As an additional contribution we give a variant of their proof, which is rather short.
The main ingredient is McMullen's upper bound theorem for classical convex polytopes \cite{McMullen:1970}, which comes in by lifting to real Puiseux series.
The special case of monomial tropical cones can also be derived from~\cite[Theorem 6.3]{BayerPeevaSturmfels:1998} or \cite{KaplanEtAl:2008}.
The latter additionally shows that, for fixed $d$, that upper bound can actually be attained, at least asymptotically.
This has already been applied in multicriteria optimization~\cite{KLV2015}, and it is used to determine the complexity of the algorithm. 
The Section~\ref{sec:computing-nondominated-set} is devoted to describing our main algorithm, Algorithm~\ref{algo:non-dominated-set}, which computes the nondominated set of a discrete multicriteria optimization problem.
This provides a new point of view on the redundancy elimination and redundancy avoidance schemes developed in~\cite{PrzybylskiGandibleuxEhrgott:2010,KLV2015}.
We end that section with a complexity analysis and one complete example arising from a multicriteria knapsack problem.
The paper closes with Section~\ref{sec:concluding}, which contains a few remarks concerning the relationship of our results with topics in commutative algebra and some open problems.
There is an established connection between discrete optimization and commutative algebra; e.g., see~\cite{DeLoeraHemmeckeKoeppe:2013}.
Hence it seems promising to study possible applications of our algorithm to topics in algebra, but this is beyond the scope of the present paper.

\section{Monomial Tropical Cones} \label{sec:tropical}
\noindent
The \emph{\tmi-tropical semiring} is the set $\TTmin=\RR\cup\{\infty\}$ equipped with $\min$ and $+$ as its addition and multiplication, respectively.
Several classical notions from linear algebra and convexity have analogs over $\TTmin$.
We introduce a special class of tropical cones which arise naturally in multicriteria optimization.
As their most important feature they admit a special kind of duality, which is not present in general tropical cones.

\subsection{Generators and tropical halfspaces}
Throughout the following we fix an integer $d\geq 1$.
A \emph{\tmi-tropical cone} $C$ is a nonempty subset of $\TTmin^{d+1}$ which is closed with respect to taking \tmi-tropical scalar combinations, i.e.,
\[
\bigl( \min( \lambda + x_0, \mu + y_0 ), \dots, \min( \lambda + x_d, \mu + y_d ) \bigr) \, \in \, C \quad \text{for all } \lambda,\mu\in \TTmin \text{ and } x,y\in C \enspace .
\]
It follows that any \tmi-tropical cone contains the point $(\infty,\infty,\dots,\infty)$.
Notice that we take indices $0,1,\dots,d$ for vectors in $\TTmin^{d+1}$.
A set $G\subset\TTmin^{d+1}$ is said to \emph{generate} the \tmi-tropical cone $C$ if this is the smallest \tmi-tropical cone which contains $G$.
\emph{Scaling} the generators tropically, i.e., adding multiples of the all-ones-vector $\1$ does not change the tropical cone.
If $C$ is finitely generated, then there is a generating set which is minimal with respect to inclusion; and this is unique, up to tropical scaling; cf.~\cite[Thm.~3.3.9]{Butkovic10}.
The elements of that minimally generating set are the \emph{extremal generators} of $C$.

Let $a$ be a vector in $\TTmin^{d+1}$.
The set $\supp(a)=\smallSetOf{i}{a_i\neq\infty}$ is its \emph{support}.
For disjoint nonempty subsets $I,J\subset\supp(a)$ the set
\begin{equation}\label{eq:halfspace}
  \SetOf{x\in\TTmin^{d+1}}{\min(x_i+a_i \mid i\in I) \leq \min(x_j+a_j \mid j\in J)}
\end{equation}
is the \emph{closed \tmi-tropical halfspace} in $\TTmin^{d+1}$ of \emph{type} $(I,J)$ with \emph{apex} $-a$.
Each \tmi-tropical halfspace is an example of a \tmi-tropical cone.
The following basic result was proved by Gaubert \cite{Gaubert:PhD}; see also \cite{Tropical+halfspaces} and~\cite{GaubertKatz07}.
It is the tropical analog of the ``Main Theorem for Cones'', as it is called in \cite[Thm.~1.3]{Ziegler:Lectures+on+polytopes}.

\begin{theorem}\label{thm:cones}
  Let $C$ be a \tmi-tropical cone which is finitely generated.
  Then $C$ is the intersection of finitely many closed \tmi-tropical halfspaces.
  Conversely, each such finite intersection is a finitely generated tropical cone.
\end{theorem}

By now various proofs for Theorem~\ref{thm:cones} are known, most of which are, in fact, constructive.
For instance, the \emph{tropical double description method} provides an algorithm \cite{DoubleDescription:2010}.
The  \emph{\tmi-tropical unit vectors} $e^{(0)}, e^{(1)}, \ldots, e^{(d)} \in \TTmin^{d+1}$ are defined by
\[
e^{(i)}_{k} \ = \
\begin{cases}
  0  & \mbox{ if }  i = k \\
  \infty & \mbox{ otherwise }
\end{cases}
\qquad \text{for } 0 \leq i,k \leq d \enspace .
\]
We set
\[
\minunit \ = \ \bigl\{ e^{(1)}, e^{(2)}, \dots, e^{(d)} \bigr\} \ \subseteq \ \TTmin^{d+1} \enspace . 
\]
Observe that the $0$th tropical unit vector is omitted.

\begin{example}\label{exmp:sector}
  Let $a\in\TTmin^{d+1}$ be a point with support $\supp(a)=\{0\}\sqcup J$ for $J$ not empty.
  The \tmi-tropical halfspace
  \begin{equation}\label{eq:sector}
    \SetOf{x\in\TTmin^{d+1}}{x_0+a_0 \leq \min(x_j+a_j \mid j\in J)}
  \end{equation}
  has type $(\{0\},J)$.
  A nonredundant set of generators is given by the $d+1$ points $g^{(0)}, g^{(1)}, \ldots, g^{(d)}$ where
  \[
  g^{(i)}_k  \ =  \
  \begin{cases}
    0 & \mbox{ if } k = 0 \\
    a_0-a_k  & \mbox{ if }  k = i\\
    \infty & \mbox{ otherwise } \enspace ,
  \end{cases}
  \]
  for $i\in J$ and $g^{(i)}=e^{(i)}$ for $i\not\in J$.
  The intersection of the \tmi-tropical halfspace~\eqref{eq:sector} with~$\RR^{d+1}$, which is a subset of $\TTmin^{d+1}$, is convex in the ordinary sense.
\end{example}

Dual to $\TTmin$ is the \emph{\tma-tropical semiring} $\TTmax=\RR\cup\{-\infty\}$, which is equipped with the operations $\max$ and $+$.
Replacing min by max in all of the above leads to \tma-tropical cones, \tma-tropical halfspaces etc.
Due to the equality
\begin{equation}\label{eq:minmax}
  \min(-x,-y) \ = \ -\max(x,y)
\end{equation}
the map $x\mapsto -x$ from $\TTmin$ to $\TTmax$ is an isomorphism of semirings.
Note that the apex $-a$ of the \tmi-tropical halfspace \eqref{eq:halfspace} lies in $\TTmax^{d+1}$.
We let $\maxunit = -\minunit$, which is contained in $\TTmax^{d+1}$.

\begin{remark}
  In the sequel we will use some very mild topological notions.
  Thus we need to briefly sketch the setup.
  The real vector space $\RR^{d+1}$ is equipped with its natural Euclidean topology.
  One way of constructing this topology is via the order topology on the reals and taking products.
  The order topology is also defined on $\TTmin$, where the open intervals form a subbasis, and this extends the order topology on $\RR$.
  With $\RR\1 = \{\lambda \cdot (1,\dots,1) \colon \lambda \in \RR\}$, it ensues that the quotient topology on
  \begin{equation}\label{eq:tp}
    \bigl(\TTmin^{d+1}\setminus\{(\infty,\dots,\infty)\}\bigr)/\RR\1
  \end{equation}
  is compact. 
  In fact, the pair $((\TTmin^{d+1}\setminus\{(\infty,\dots,\infty)\})/\RR\1,\RR^{d+1}/\RR\1)$ is homeomorphic with the pair $(\Delta_d,\topint\Delta_d)$, where $\Delta_d$ is the $d$-dimensional (standard-)simplex, and $\topint\Delta_d$ is its interior.
  The topological space \eqref{eq:tp} is the \emph{$d$-dimensional tropical projective space} with respect to $\min$.
  Clearly, exchanging $\min$ by $\max$ essentially gives the same.
  Note, however, that $\TTmin^{d+1}$ and $\TTmax^{d+1}$ differ as sets, with $\RR^{d+1}$ as their intersection.
  See \cite{JoswigLoho:2016} for more details on tropical convexity in the tropical projective space.
\end{remark}

\begin{example}\label{exmp:sector-as-max}
  The intersection $S$ of the \tmi-tropical halfspace \eqref{eq:sector} with $\RR^{d+1}$ is convex in the ordinary sense.
  It follows from \cite[Prop.~48]{JoswigLoho:2016} that $S$ is a `weighted digraph polyhedron' and thus the topological closure of $S$ in $\TTmax^{d+1}$ is a \tma-tropical cone.
  That \tma-tropical cone admits an exterior description in $\TTmax^{d+1}$, and thus
  \[
  S \ = \ \bigcap_{i\in\supp(a) \setminus \{0\}} \SetOf{x \in \RR^{d+1}}{x_0 + a_0\leq x_i + a_i}
  \]
  is the intersection of finitely many \tma-tropical halfspaces in $\RR^{d+1}$.
  The $d+1$ points in the set $\{-a\}\cup\maxunit$ form the extremal generators of the closure of $S$ in $\TTmax^{d+1}$.
  This is a \tma-tropical cone arising as the intersection of the sets $\SetOf{x \in \TTmax^{d+1}}{x_0 + a_0\leq x_i + a_i}$ for $i\in\supp(a) \setminus \{0\}$. 
\end{example}

Let $G \subseteq \TTmax^{d+1}$ be finite such that $0$ is contained in the support of each point.
We define 
\[
\closedmonomial{G} \ = \ \bigcup_{g \in G} \SetOf{x\in\TTmax^{d+1}}{x_0-g_0 \leq \min(x_j-g_j \mid j\in \supp(g)\setminus\{0\})} \enspace ,
\]
and let $\monomial{G}=\closedmonomial{G}\cap\RR^{d+1}$.
By construction the latter set is a finite union of the \tmi-tropical halfspaces intersected with $\RR^{d+1}$ studied in Examples~\ref{exmp:sector} and~\ref{exmp:sector-as-max}.
See also Figure~\ref{fig:complementary-cones} below.

\begin{lemma}\label{lem:monomial-maxcone}
  The set $\closedmonomial{G}$ is the \tma-tropical cone in $\TTmax^{d+1}$ generated by the finite set $G\cup\maxunit$.
\end{lemma}
\begin{proof}
  Let $C$ denote the \tma-tropical cone generated by $G\cup\maxunit$, and we want to show that $C$ agrees with $\closedmonomial{G}$.
  We have $G\cup\maxunit\subseteq\closedmonomial{G}$, and it follows from Example~\ref{exmp:sector-as-max} that $\closedmonomial{G}$ is a subset of $C$.

  For the reverse inclusion we need to show that $\closedmonomial{G}$ is a \tma-tropical cone in $\TTmax^{d+1}$.
  To this end consider $g,h\in G$ distinct.
  Let $x\in \closedmonomial{g}$, $y\in \closedmonomial{h}$ and $\lambda,\mu\in\TTmax$.
  Without loss of generality we may assume that $\lambda + x_0 \geq \mu + y_0$.
  But this entails that the \tma-tropical linear combination
  \begin{equation}\label{eq:max-linear-combination}
    \bigl( \max( \lambda + x_0, \mu + y_0 ), \dots, \max( \lambda + x_d, \mu + y_d ) \bigr)
  \end{equation}
  satisfies the inequality
  \begin{equation*}
  \begin{split}
    \lambda + x_0 - g_0 \ &\leq \ \min(\lambda+x_j-g_j \mid j\in \supp(g)\setminus\{0\}) \\
    &\leq \ \min\bigl( \max(\lambda+x_j,\mu+y_j)-g_j \mid j\in \supp(g)\setminus\{0\}\bigr) \enspace .
  \end{split}
  \end{equation*}
  That is, the \tma-tropical linear combination \eqref{eq:max-linear-combination} is contained in $\closedmonomial{g}$, and this proves our claim.
\end{proof}
Observe that
\begin{equation}\label{eq:orthant}
  \monomial{0} \ = \ \bigl(\{0\} \times \RR_{\geq 0}^d\bigr) + \RR\1 \enspace .
\end{equation}
Identifying $\RR^d$ with $\smallSetOf{x\in\RR^{d+1}}{x_0=0}$ allows to view $G\subset\NN^d$ as a subset of $\RR^{d+1}$.
In this way the integral points in $\monomial{G}$ with zero first coordinate correspond to the set of monomials in the monomial ideal generated by $G$; see \cite{MillerSturmfels:2005}, \cite {HerzogHibi:2011} and also Section~\ref{sec:concluding} below.
For this reason we call the set $\closedmonomial{G}$ the \emph{monomial \tma-tropical cone} generated by $G$, and $\monomial{G}$ is its \emph{real part}.

\begin{remark}
  We defined monomial tropical cones by giving the $0$th coordinate a special role.
  However, the space $\TTmin^{d+1}$ is symmetric with respect to permuting coordinates, which is why that particular choice of the coordinate is inessential.
  This means that there is a natural notion of an \emph{$i$-monomial \tmi-tropical cone} which generalizes the above.
  Furthermore, since the single inequality in~\eqref{eq:halfspace} is equivalent to the system
  \[
  \min(x_i+a_i \mid i\in I) \ \leq \ x_j+a_j \quad \mbox{ for each } j\in J
  \]
  of at most $d$ inequalities, it follows that each tropical cone can be written as the intersection of at most $d+1$ monomial tropical cones, one for each $i \in [d] \cup \{0\}$.
  The $i$-monomial tropical cones are precisely the `$i$th polar cones' in \cite{AllamigeonGaubertKatz:2011}.
  We propose a different name to stress the connection to commutative algebra; cf.\ Section~\ref{sec:concluding} below.
\end{remark}

\begin{lemma} \label{lem:interior-union-sectors}
  The interior of the real part of the monomial \tma-tropical cone $\monomial{G}$ in $\RR^{d+1}$ for $G \subseteq \TTmax^{d+1}$ equals
  \begin{equation}\label{eq:interior}
    \bigcup_{g \in G} \SetOf{x\in\RR^{d+1}}{x_0-g_0 < \min(x_j-g_j \mid j\in \supp(g)\setminus\{0\})} \enspace .
  \end{equation}
\end{lemma}
\begin{proof}
  For a point $z$ in $\RR^{d+1}$ let
  \[
  G(z) \ = \ \SetOf{g\in G}{z\in \monomial{g}} \enspace .
  \]
  Now consider a point $z$ in $\monomial{G}$.
  If $z$ lies in the boundary of $\monomial{g}$ for each $g\in G(z)$ then $z - \epsilon (0,1,\ldots,1)$ is not contained in $\monomial{G}$ for every $\epsilon > 0$.
  This implies that if $z$ is an interior point then it must be contained in the interior of $\monomial{g}$ for at least one $g\in G(z)$.
\end{proof}

Observe that, by construction, the set \eqref{eq:interior} is a \tma-tropically convex set which is open. We call it the \emph{interior} of $\closedmonomial{G}$. Here, a set $S$ is \tma-tropically convex if $\max(x,\lambda\cdot \1 + y) \in S$ for all $x,y \in S$ and $\lambda \leq 0$, where the $\max$ is taken componentwise.

\begin{lemma}\label{lem:monomial-halfspace}
  Each \tma-tropical halfspace which contains the monomial \tma-tropical cone $\closedmonomial{G}$ has type $(\{0\},J)$ for some nonempty set $J\subset\{1,2,\dots,d\}$.
  Equivalently, the intersection of such a tropical halfspace with $\RR^{d+1}$  has the form
  \begin{equation}\label{eq:monomial-halfspace}
    \SetOf{x\in\RR^{d+1}}{x_0-a_0 \leq \max(x_j-a_j \mid j\in J)}
  \end{equation}
  for some point $a\in\TTmin^{d+1}$ with $\supp(a)=\{0\}\sqcup J$.
\end{lemma}
\begin{proof}
  By definition, the monomial tropical cone $\closedmonomial{G}$ contains the rays
  \[
  x + \RR_{\geq 0}^{d+1}\cdot(0,\,\underbrace{0,\dots,0}_{j-1},\, 1,\, \underbrace{0,\dots,0}_{d-j} )
  \]
  for all $x\in \monomial{G}$ and $1\leq j\leq d$, but not the ray $x + \RR_{\geq 0}^{d+1}\cdot(1,0,\dots,0)$.
  This implies that no variable indexed by $1, 2, \ldots, d$ can occur on the left hand side of an inequality of the form~\eqref{eq:halfspace}, and the claim follows. 
\end{proof}

Let us define $\complementarymonomial{G}$ as the closure of the complement of $\monomial{G}$ in $\RR^{d+1}$.
Further we let $\closedcomplementarymonomial{G}$ be its closure in $\TTmin^{d+1}$ with respect to min.
Note that seeing $\closedmonomial{G}$ and $\closedcomplementarymonomial{G}$ simultaneously requires to go to the union $\TTmin^{d+1}\cup\TTmax^{d+1}$.

\begin{example} \label{ex:complementary-cones}
  Consider the points
\begin{align*}
a = ( 0,1,0), \quad b=(0,0,2), \quad c=(0,-3,\infty)\quad \mbox{ in } & \TTmin^3 \quad \mbox{ and } \\
f = ( 0,1,-\infty), \quad g= (0,0,0), \quad h=(0,-3,2)\quad \mbox{ in } & \TTmax^3 \enspace .
\end{align*}
Then $\closedmonomial{f,g,h}$ is the \tma-tropical cone generated by $\{f,g,h,-e^{(1)},-e^{(2)}\}$. Further, $\closedcomplementarymonomial{a,b,c}$ is the \tmi-tropical cone generated by $\{a,b,c,e^{(1)},e^{(2)}\}$. 
The real part $\RR^3$ decomposes into $\complementarymonomial{f,g,h}=-\monomial{-a,-b,-c}$ and $\monomial{f,g,h}$, where the intersection is homeomorphic to $\RR^{2}$.
Those two cones are shown in Figure~\ref{fig:complementary-cones}. Note that the monomial tropical cones always contain the whole line $x + \RR\1$ for each point $x$.
This allows to flatten the picture to $\RR^2$ by choosing the representative of $x$ in $x + \RR\1$ with $0$th coordinate $0$.
This is convenient, but the behavior at infinity is somewhat hard to visualize since in the flattened picture, e.g., $(\infty,\infty,0)$ and $(0,1,-\infty)$ lie in the same direction.
\end{example}

\begin{figure}[htb]
  \centering
  \begin{tikzpicture}[scale=1.1]

  \fillbox{-4}{-5}{4}{5};
  \boundingbox{-\inflen}{-\inflen}{4}{5}{\RR^3};
  \boundingbox{-4}{-5}{\inflen}{\inflen}{};

  \node[Apex1, label=left:{$g = (0,0,0)$}] (g) at (0,0) {};
  \node[Apex1, label=above right:{$h = (0,-3,2)$}] (h) at (-3,2) {};
  \node[Apex1, label=right:{$f = (0,1,-\infty)$}] (f) at (1,-5.5) {};
    
  \node[ConeLabel] at (2.6,3.3) {$\monomial{f,g,h}$};

  \node[Apex2, label=below right:{$b=(0,0,2)$}] (b) at (0,2) {};
  \node[Apex2, label=below right:{$a=(0,1,0)$}] (a) at (1,0) {};
  \node[Apex2, label={[fill=white]left:{$c=(0,-3,\infty)$}}] (c) at (-3,5.5) {};

  \node[ConeLabel, text width = 2.5cm] at (-2,-3.3) {$\complementarymonomial{f,g,h} = -\monomial{-a,-b,-c}$};

  \node[InfGen, label=below right:{$(-\infty,-\infty,0)$}] (emax2) at (0, \inflen) {};
  \node[InfGen, label={[fill=white]below:{$(-\infty,0,-\infty)$}}] (emax1) at (\inflen,0) {}; 
  \node[InfGen, label=below right:{$(\infty,0,\infty)$}] (emin1) at (-\inflen,0) {};
  \node[InfGen, label=below left:{$(\infty,\infty,0)$}] (emin2) at (0, -\inflen) {};

  \draw[Border] (c) -- (h) -- (b) -- (g) -- (a) -- (f);
  
  \node [label={above right:{$\TTmin^3$}}] at (-\inflen,-\inflen) {};
  \node [label={below left:{$\TTmax^3$}}] at (\inflen,\inflen) {};
  
\end{tikzpicture}

  \caption{Complementary pair of monomial tropical cones}
  \label{fig:complementary-cones}
\end{figure}
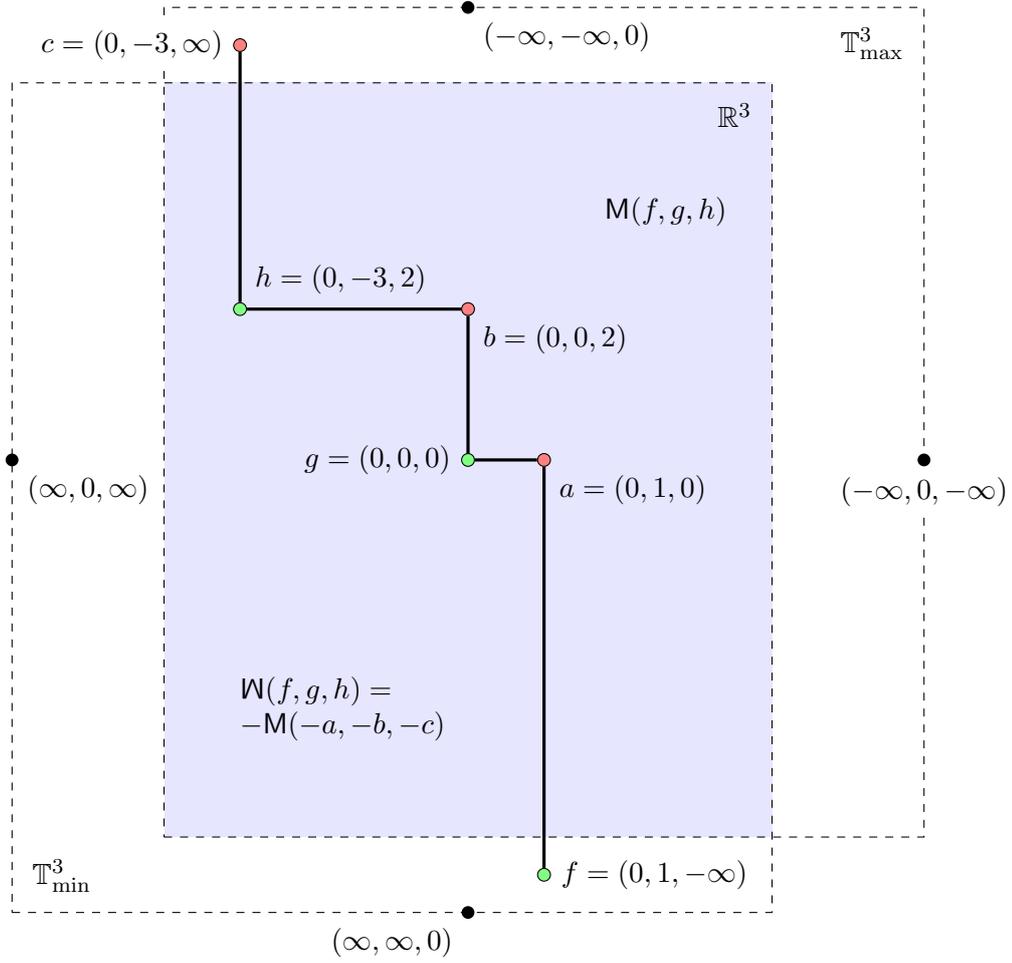

The following result extends the structural insight of \cite[Theorem 4]{AllamigeonGaubertKatz:2011}.

\begin{theorem} \label{thm:complementary-cones}
  The set $\closedcomplementarymonomial{G}$ is a \tmi-tropical cone.
  More precisely, if $\cH$ is a set of \tma-tropical halfspaces such that $\bigcap \cH=\closedmonomial{G}$, then
  \[
  \complementarymonomial{G} \ = \ - \monomial{-A} \enspace ,
  \]
  where $A \subset \TTmin^{d+1}$ is the set of apices of the tropical halfspaces in $\cH$.
 In particular, the set $A\cup\minunit$ generates $\closedcomplementarymonomial{G}$.
\end{theorem}
\begin{proof}
  An application of De Morgan's Law shows that the complement in $\RR^{d+1}$ of the real part $\monomial{G}$ is the interior of a \tmi-tropical cone
  \begin{equation} \label{eq:exterior-complementary}
    \bigcap_{g \in G} \SetOf{x\in\RR^{d+1}}{x_0-g_0 > \min(x_j-g_j \mid j\in \supp(g)\setminus\{0\})} \enspace .
  \end{equation}
  It follows that the closure $\closedcomplementarymonomial{G}$ in $\TTmin^{d+1}$ is a closed \tmi-tropical cone.
  For each $a\in A$, by Lemma~\ref{lem:monomial-halfspace}, the corresponding \tma-tropical halfspace is of type $(\{0\},J)$, and it looks like \eqref{eq:monomial-halfspace}.
  Since
  \begin{equation} \label{eq:real+mono+cone+intersection+halfspaces}
  \monomial{G} \ = \ \bigcap_{a\in A} \SetOf{x\in\RR^{d+1}}{x_0-a_0 \leq \max(x_j-a_j \mid j\in \supp(a)\setminus\{0\})}
  \end{equation}
  we get
  \begin{equation} \label{eq:real+mono+cone+union+orthants}
  \complementarymonomial{G} \ = \  \bigcup_{a \in A} \SetOf{x\in\RR^{d+1}}{x_0-a_0 \geq \max(x_j-a_j \mid j\in \supp(a)\setminus\{0\})} \enspace .
  \end{equation}

  To obtain~\eqref{eq:real+mono+cone+union+orthants} from~\eqref{eq:real+mono+cone+intersection+halfspaces}, we take the closure of the complement in $\RR^{d+1}$.
  Here, the finiteness of $A$ allows us to take the closure of each tropical halfspace with apex in $A$ separately.
  Now, in view of the equality \eqref{eq:minmax}, we have that $x_0-a_0 \geq \max(x_j-a_j \mid j\in \supp(a)\setminus\{0\}) = -\min(-x_j+a_j \mid j\in \supp(a)\setminus\{0\})$, and, equivalently,
  \begin{equation} \label{eq:reversed-inequality}
    -x_0 + a_0 \ \leq \ \min(-x_j+a_j \mid j\in \supp(a)\setminus\{0\}) \enspace .
  \end{equation}
  By applying Lemma~\ref{lem:monomial-maxcone} to $\closedmonomial{-A}$ we see that $\closedmonomial{-A}$ is the \tma-tropical cone in $\TTmax^{d+1}$ generated by $-A \cup \maxunit$.
  Combining this observation with \eqref{eq:reversed-inequality} it follows that the real part of the \tmi-tropical cone $\complementarymonomial{G}$ equals $-\monomial{-A}$.
  Hence, the \tmi-tropical cone $\closedcomplementarymonomial{G}$ is generated by $A \cup \minunit$.
\end{proof}
We call $\closedcomplementarymonomial{G}$ the \emph{complementary monomial $min$-tropical cone} of the monomial \tma-tropical cone $\closedmonomial{G}$, and denote by $\complementarymonomial{G}$ its \emph{real part}. 

\begin{corollary}\label{cor:complement}
  The \tmi-tropical halfspaces with apices in the set $G$ yield an exterior description of the \tmi-tropical cone $\closedcomplementarymonomial{G}$.
  Its set $A$ of extremal generators satisfies
  \[
  \RR^{d+1} \setminus \monomial{G} \ =\ \bigcup_{a \in A} \SetOf{x\in\RR^{d+1}}{x_0-a_0 > \max(x_j-a_j \mid j\in \supp(a)\setminus\{0\})} \enspace .
  \]
\end{corollary}
\begin{proof}
  The first claim follows by taking the closure in Equation~\ref{eq:exterior-complementary}. The second claim follows from $\complementarymonomial{G} = - \monomial{-A}$ with Lemma~\ref{lem:interior-union-sectors}.
\end{proof}

\begin{remark}
  A set $Y\in\RR^d$ is `$\RR^d_{\geq 0}$-convex' according to \cite[Definition 3.1]{Ehrgott:2005} if $Y+\RR^d_{\geq 0}$ is convex in the ordinary sense.
  Each ordinary convex set is `$\RR^d_{\geq 0}$-convex', but the converse is false.
  There is no direct relationship with tropical convexity. 
  The latter is just the \tma-tropical analogue of a convex combination.
  For instance, the set
  \[
  \bigl\{ (1,0),\, (0,1) \bigr\} + \RR^2_{\geq 0}
  \]
  is \tma-tropically convex but it is not `$\RR^2_{\geq 0}$-convex'.
  Conversely, the unit disk in $\RR^2$ is convex in the ordinary sense, and thus `$\RR^2_{\geq 0}$-convex', but it is not \tma-tropically convex.
\end{remark}

The remainder of this section is devoted to describing the various algorithmic contributions from tropical convexity to our Algorithm~\ref{algo:non-dominated-set} (given in the section below) for computing the nondominated set of a discrete multicriteria optimization problem.
We believe that these observations are also of independent interest.
The first result in this direction exhibits a dichotomy which is similar in spirit to the Farkas Lemma of linear programming.

\begin{lemma}\label{lem:farkas}
  Suppose that $G\subset\TTmax^{d+1}$ and $G'\subset G$.
  Let $a \in \RR^{d+1}$ be an extremal generator of $\closedcomplementarymonomial{G'}$.
  Then either $a$ is an extremal generator of $\closedcomplementarymonomial{G}$ or there is a point $g \in G$ such that $g \in a - \left(\bigl(\{0\} \times \RR_{> 0}^d\bigr) + \RR\1\right)$.
\end{lemma}

\begin{proof}
  An exterior description of $\closedcomplementarymonomial{G}$ is given in Corollary~\ref{cor:complement}.
  The point $a$ is an extremal generator of $\closedcomplementarymonomial{G}$ if and only if it does not separate a generator of $\closedcomplementarymonomial{G}$.
  This means that there is no $g \in G$ such that
  \[
  a \in \SetOf{x\in\RR^{d+1}}{x_0 - g_0 < \min(x_j-g_j \mid j\in [d])} \enspace .
  \]
  With \eqref{eq:orthant} this implies the claim.
\end{proof}

For exploiting Lemma~\ref{lem:farkas} in our Algorithm~\ref{algo:non-dominated-set} we need a method to filter out the extremal generators of those tropical cones that arise in our procedure.

To embed it in the larger context, we recall refinements of the classical double description method from Fukuda and Prodon in~\cite{FukudaProdon:1996}. Their Lemma~3 describes the general method for computing a set containing all generators of the updated cone (including many redundant ones). The tropical analog for monomial tropical cones is described in Lines~\ref{line:pairs-in-out} to~\ref{line:min} of Algorithm~\ref{algo:newextremals}. Fukuda and Prodon describe the computation of only non-redundant generators in~\cite[Lemma~8]{FukudaProdon:1996}, which involves the characterization of adjacent pairs in advance based on~\cite[Proposition~7]{FukudaProdon:1996}.
The adjacency structure developed in~\cite{DaechertEtAl} is very close to the classical adjacency described in~\cite{FukudaProdon:1996}. This is a strengthened approach to the redundancy avoidance in~\cite{KLV2015}.
An earlier filtering method is redundancy elimination. This technique was already used in~\cite{PrzybylskiGandibleuxEhrgott:2010}. An improved version is presented in~\cite{KLV2015}. Additionally, Line~6 of Algorithm~2 in \cite{KLV2015} shows that, for monomial tropical cones, not even all the pairs from $A^{\geq} \times (A \setminus A^{\geq})$ are required but the pairs in $\minunit \times A \setminus A^{\geq}$ are enough. 
As remarked in \cite[Corollary 8]{AllamigeonGaubertKatz:2011}, alternatively, an adaptation of \cite{BEGKM:2002} yields an incremental quasi-polynomial time algorithm based on a generalization of the minimal hypergraph transversal generation.

%This is a simple geometric property and corresponds to the redundancy elimination in \cite{KLV2015}.

Our filtering is not tailored to be as efficient but provides the geometric intuition from tropical convexity. We offer two methods in Lemma~\ref{lem:check-extremal} and Lemma~\ref{lem:inner-extremal} which are meant to fill in for Lines~\ref{line:check-extremal} --~\ref{line:nu} in the algorithm.
They are specializations of known more general methods, which are adapted to tropical cones which are monomial; cf.\ \cite[Proposition 2.4]{DaechertEtAl}.

\begin{lemma}\label{lem:check-extremal}
  The point $a \in \RR \times \TTmin^d$ is an extremal generator of $\closedcomplementarymonomial{G}$ if and only if, for every $j \in \supp(a)\setminus\{0\}$, there is an apex $g \in G$ such that
  \begin{equation} \label{eq:extremal-index}
  a_0 - g_0 = \min(a_{\ell} - g_{\ell} \mid \ell\in\supp(a)\setminus\{0\})  = a_j - g_j < a_{\ell} - g_{\ell} \enspace ,
  \end{equation}
 for each $\ell \in \supp(a)\setminus\{0,j\}$.
\end{lemma}
\begin{proof}
  Monomial tropical cones occur as `$i$th polar tropical cones' in \cite{AllamigeonGaubertKatz:2011}.
  Our claim follows from their Theorem~3.
\end{proof}

\begin{proposition}\label{prop:newextremals}
  For $G\subset\TTmax^{d+1}$ and $h\in\TTmax^{d+1}$ with $h_0=0$ the Algorithm~\ref{algo:newextremals} correctly returns the extremal generators of $\closedcomplementarymonomial{G \cup \{h\}}$.
\end{proposition}
\begin{proof}
  The method is an adaptation of the procedure \textproc{ComputeExtreme} in \cite{DoubleDescription:2010} to monomial tropical cones.
  The restriction to a special class of tropical cones allows us to simplify the expressions from the Algorithm in~\cite[Figure 9]{DoubleDescription:2010}.
  There is no need to calculate with proper two-sided inequalities as one side is just a single term.
  Apart from that, the algorithmic structure is just the same.
  % By the discussion below the theorem, the computation of this function takes $\mathcal{O}(|H|\cdot d \alpha(d) |V|^2)$ where $\alpha$ is the inverse Ackermann-function.
  Hence, its correctness follows directly from \cite[Theorem 4.1]{DoubleDescription:2010}.
  Note that a possible extremality test in line \ref{line:check-extremal} is given by Lemma~\ref{lem:check-extremal}.
\end{proof}

\begin{algorithm}[htbp]
\caption{Extremal generators of a monomial tropical cone}  \label{algo:newextremals}
  \begin{algorithmic}[1]
    \Require{A set $G\subset\TTmax^{d+1}$, the set $A$ of extremal generators of $\closedcomplementarymonomial{G}$, and a point $h\in\TTmax^{d+1}$ with $h_0=0$.}
    \Ensure{The set of extremal generators of $\closedcomplementarymonomial{G \cup \{h\}}$.}
    \Procedure{NewExtremals}{$G,A,h$}
    \State $A^{\geq} \gets \SetOf{a \in A}{a_0 \geq \min_{i \in [d]}(a_i - h_i)}$ 
    \State $B \gets A^{\geq}$
    \For{all pairs of $b \in A^{\geq}$ and $c \in  A \setminus A^{\geq}$ } \label{line:pairs-in-out}
    %\State $v \gets ((-h)\odot u)\odot u \oplus w$
    \State $\lambda \gets -\min_{i \in [d]}(b_i - h_i)$
    \State $a \gets \min(\lambda\1_{d+1} + b , c)$ \label{line:min} 
    \If{$a$ extremal in $\closedcomplementarymonomial{G \cup \{h\}}$} \label{line:check-extremal}
    \State $B \gets B \cup \{\nu(a)\}$ \label{line:nu}
    \EndIf
    \EndFor
    \State \Return $B$ 
    \EndProcedure
  \end{algorithmic}
\end{algorithm}

Note that the operation `$\min$' in line \ref{line:min} of Algorithm~\ref{algo:newextremals} is component-wise.
That is, it takes two points $x,y\in\TTmin^d$ as arguments and returns the point $(\min(x_1,y_1),\dots,\min(x_d,y_d))$.
The operation $\nu(a)$ in line \ref{line:nu} picks the unique representative $(0,a_1-a_0,\dots,a_d-a_0)$ in $\RR\1+a$ with leading coefficient zero.
Observe that $a_0<\infty$ since $c_0<\infty$.

Next we give an alternative to Lemma~\ref{lem:check-extremal} and thus for line \ref{line:check-extremal} of Algorithm~\ref{algo:newextremals}.

\begin{lemma}\label{lem:inner-extremal}
  Let $A \subseteq \RR^{d+1}$ be a set of finite generators of a \tmi-tropical monomial cone.
  A generator $b \in A$ is extremal if and only if there is no $a \in A \setminus \{b\}$ with
  \[
  a_0 - b_0 \ \leq \ \min(a_1 - b_1, \ldots, a_d - b_d) \enspace .
  \]
\end{lemma}
\begin{proof}
  This is a special case of \cite[Prop.~3.3.6]{Butkovic10}. It can be applied as Theorem~\ref{thm:complementary-cones} transfers the role of the $0$th coordinate from the inequality description to the generators.
  The claim follows from Lemma~\ref{lem:monomial-maxcone}, as $-\closedmonomial{A}$ is just the union of the sets $-\closedmonomial{a}$ for $a \in A$. 
\end{proof}

The main difference to the above is that Lemma~\ref{lem:inner-extremal} does not make use of an exterior description of the tropical cone.
It directly translates into an algorithm of (unit cost) complexity $O(n^2d)$ for determining all extremal generators of a monomial tropical cone generated from $n$ points.
For practical applications of multicriteria optimization it may be useful to explore which of the two methods is superior in a given scenario.

\section{An Upper Bound Theorem}\label{sec:upper-bound}
\noindent
Now we will study the combinatorial complexity of general tropical cones.
In the subsequent section this will be used to recover known bounds on the number of scalarizations required for finding the nondominated points of a given multicriteria optimization problem.
We start out with relating tropical cones with ordinary convex cones defined over a suitably chosen ordered field.
A formal power series in $t$ of the form
\[
\gamma(t) \ = \ \sum_{u\in U} c_u t^u
\]
is a \emph{generalized real Puiseux series} if
\begin{inparaenum}[(i)]
\item the set $U$ of exponents is a countable and well-ordered subset of~$\RR$,
\item which is finite, or it has $\infty$ as its only accumulation point, and
\item the coefficients $c_u$ are real numbers.
\end{inparaenum}
%  Note that this implies that each sequence in $U$ converges to $\infty$.
The usual Puiseux series have rational exponents with a common denominator.
The \emph{valuation map} $\val$ sends a (generalized) Puiseux series to its lowest exponent.
Further, the real (generalized) Puiseux series are ordered: the sign is given by the sign of the coefficient of the term of lowest order.
It is a consequence of \cite[Theorem~1]{Markwig:2010} that the set $\hahnseries{\RR}{t}$ of generalized real Puiseux series, equipped with the coefficient-wise addition and the usual convolution product, forms a real closed field.
By the Tarski--Seidenberg principle \cite[Theorem 2.80]{BasuPollackRoy:2006} the first-order theories of the ordered fields $\hahnseries{\RR}{t}$ and $\RR$ coincide.
In particular, convexity, linear programming and polyhedra work like over the reals.

The connection to tropical convexity comes from the following observation \cite[\S2]{DevelinYu:2007}.
The valuation map $\val:\hahnseries{\RR}{t}\to\TTmin$ can be extended coordinatewise and pointwise to arbitrary subsets of the vector space ${\hahnseries{\RR}{t}}^{d+1}$.
As a key fact the restriction of $\val$ to the nonnegative elements of $\hahnseries{\RR}{t}$ is a surjective homomorphism of semirings onto $\TTmin$.
In this way, we see that for an ordinary cone $\bm C$ in ${\hahnseries{\RR}{t}}^{d+1}$ the image $\val(\bm C)$ is a \tmi-tropical cone in $\TTmin^{d+1}$.
Conversely, each \tmi-tropical cone arises in this way.

McMullen's upper bound theorem \cite{McMullen:1970} says that the maximal number of extremal generators of an ordinary polyhedral cone in $\RR^{k+1}$ with $m$ facets is bounded by
\begin{equation}\label{eq:upper-bound}
  U(m,k) \ = \ \binom{m - \lceil k/2 \rceil}{\lfloor k/2 \rfloor} + \binom{m-\lfloor k/2 \rfloor - 1}{\lceil k/2 \rceil - 1} \enspace ,
\end{equation}
which is the number of facets of a cyclic $k$-polytope with $m$ vertices.
A direct computation shows that $U(m,k)$ lies in $\Theta(m^{\lfloor k/2\rfloor})$ for $k$ fixed.
By expressing tropical cones as limits of classical cones, McMullen's upper bound theorem was used by Allamigeon, Gaubert and Katz to derive an upper bound theorem for tropical cones. 
Here we give a variation of their argument, which leads to a rather short proof.

\begin{theorem}[{Allamigeon, Gaubert and Katz \cite[Theorem 1]{AGK:ExtremePoints2011}}]\label{thm:upper-bound}
  The number of extreme rays of an arbitrary tropical cone in $\TTmin^{d+1}$ defined as the intersection of $n$ tropical halfspaces is bounded by $U(n+d,d)$.
\end{theorem}
\begin{proof}
  Let $C$ be a tropical cone given as the intersection of the tropical halfspaces $H_1, \ldots, H_n$.
  By \cite[Proposition 2.6]{ABGJ:2015}, there are halfspaces $\bm H_1, \ldots, \bm H_n$ in ${\hahnseries{\RR}{t}}^{d+1}$ with $\val(\bm H_j) = H_j$, for $j \in [n]$, such that
  \[
    \val\left( \, \bigcap_{j=1}^{n} \bm H_j \, \cap \, \bigcap_{i=1}^{d} \{\bm x_i \geq \bm 0\} \, \right) \ = \ \bigcap_{j=1}^{n} H_j \enspace ,
  \]
  and, additionally, the generators of the ordinary cone $\bm C=\bigcap \bm H_j$ are mapped onto the generators of the tropical cone $C$. 
  The ordinary cone $\bm C$ has at most $n+d$ facets, and thus the claim follows from McMullen's upper bound theorem.
\end{proof}

The fact that an additional summand of $d$ occurs in the first argument of the upper bound function $U(\cdot,\cdot)$ may be surprising at first sight.
Being able to go back and forth between cones over Puiseux series and tropical cones easily requires to restrict to nonnegative Puiseux series.
So the `$+d$' accounts for the nonnegativity constraints.

In \cite{KaplanEtAl:2008} Kaplan et al.\ study colored orthogonal range counting and implicitly give the asymptotic upper and lower bound for the special case of monomial tropical cones.
The upper bound also follows from \cite{BayerPeevaSturmfels:1998}; cf.\ Section \ref{sec:concluding}.
For a subset $G$ of $n$ points in $\RR^d$, they define an open \emph{empty orthant} as the product of open intervals $O = \prod_{i=1}^{d}(-\infty,a_i)$ where $a = (a_1, \ldots, a_d) \in \RR^d$ is chosen such that $O$ does not contain a point of $G$.
The empty orthant $O$ is \emph{maximal} if any increase in a coordinate of $a$ would yield a point of $G$ contained in $O$. By Lemma~\ref{lem:check-extremal}, the apex $a$ of a maximal empty orthant is precisely an extremal generator of $\closedcomplementarymonomial{G}$. 
With the reduction from orthants to boxes by doubling the coordinates before \cite[Theorem 2.4]{KaplanEtAl:2008}, one can apply the construction of \cite[Lemma 3.3]{KaplanEtAl:2008} to obtain a lower bound on the number of maximal empty orthants with respect to a given set~$G$.
This translates to a lower bound on the number of extremal generators of $\closedcomplementarymonomial{G}$.
In this way the argument of Kaplan et al.\ shows that the upper bound in Theorem~\ref{thm:upper-bound} is actually asymptotically tight, even for monomial tropical cones.

\begin{corollary}\label{cor:asymptotic-maximum}
  For $d$ an absolute constant, the maximal number of extremal generators of a (monomial) tropical cone in $\TTmin^{d+1}$ defined as the intersection of $n$ tropical halfspaces lies in $\Theta(n^{\lfloor d/2 \rfloor})$.
\end{corollary}

\section{Computing the nondominated set} \label{sec:computing-nondominated-set}
\noindent
We consider the multicriteria optimization problem $\min f(x)$ for $x\in X$, where $f$ is a $d$-tuple of objective functions as in \eqref{eq:multiopt}.
Our main focus lies on the outcome space $Z = f(X)$, which is a subset of $\RR^d$.
Following \cite[Table 1.2 and Def.~2.1(6)]{Ehrgott:2005} we let $w\leqq z$ if $w_i\leq z_i$ for all $i\in[d]$, and this defines a partial ordering on $\RR^d$.
For any subset $S\subset\RR^d$ the minimal elements with respect to $\leqq$ form the \emph{nondominated points}.

We say that a multicriteria optimization problem is \emph{discrete} if the nondominated set is finite and nonempty.
Note that the nondominated set can be empty even if the feasible set is not, e.g., if the feasible set is $\ZZ^d$.
Furthermore, if a problem is discrete then the \emph{ideal point} defined by the componentwise infimum of $Z$ is finite. 
In the literature, a multicriteria optimization problem is usually called `discrete' if the feasible set is finite.
In our setting the generalization is more natural.

\begin{lemma}[{\cite[Proposition 2.3]{Ehrgott:2005}}]\label{lem:ehrgott}
  For any set $S\subset\RR^d$ the nondominated set of $S$ equals the nondominated set of $S + \RR_{\geq 0}^d$.
\end{lemma}

\begin{figure}[htb]
  \centering
  \renewcommand\inflen{5}

\begin{tikzpicture}

  \fill[DominatedSet] (-3, \inflen) -- (-3,2) -- (0.1,2) -- (0.1,\inflen) -- cycle;
  \fill[DominatedSet] (0,0) -- (\inflen,0) -- (\inflen, \inflen) -- (0,\inflen);

  \fill[SearchRegion] (-\inflen, \inflen) -- (-3, \inflen) -- (-3, -1.7) -- (-\inflen, -1.7) --cycle;
  \fill[SearchRegion] (-3,2) -- (0,2) -- (0, -1.7) -- (-3, -1.7) --cycle;
  \fill[SearchRegion] (0,0) -- (\inflen,0) -- (\inflen, -1.7) -- (0, -1.7) -- cycle;
  
  \node[Apex1, label=above right:{$g = (0,0)$}] (g) at (0,0) {};
  \node[Apex1, label=above right:{$h = (-3,2)$}] (h) at (-3,2) {};

  \draw[Border] (-3, \inflen+0.2) -- (h) -- (0,2) -- (g) -- (\inflen+0.2,0);

  \node[Apex2] at (0,2) {};
  \node[Apex2] at (-3,\inflen+0.2) {};
  \node[Apex2] at (\inflen+0.2,0) {};

  \node[ConeLabel] at (3,3) {$\monomial{g,h}$};
  \node[ConeLabel] at (-2.3,-0.6) {$\complementarymonomial{g,h}$};

\end{tikzpicture}

  \caption{The real part of the monomial tropical cone $\monomial{g,h}$ from Example~\ref{ex:multi}.
    Compare with Figure~\ref{fig:complementary-cones}, which shows $\monomial{f,g,h}$.}
  \label{fig:multi}
\end{figure}
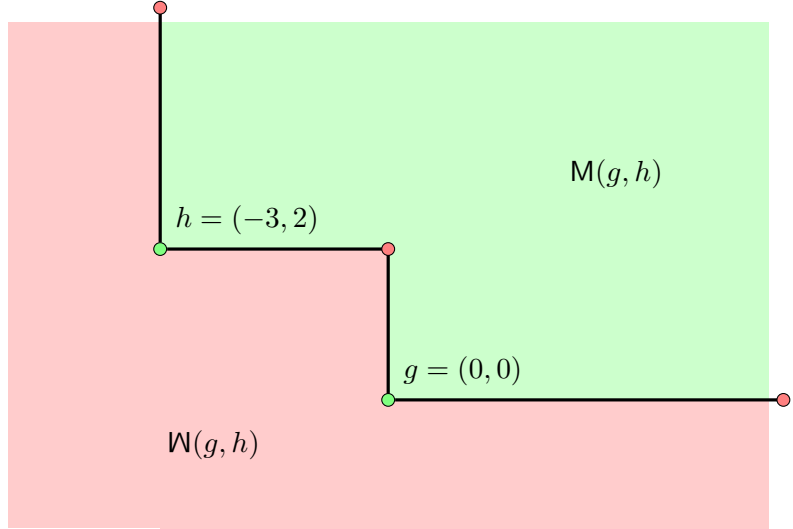

\begin{example}\label{ex:multi}
  Consider the multiobjective optimization problem given by
  \[
  \min \begin{pmatrix} -3 & 1 & 1\\ 2 & 1 & 1 \end{pmatrix} \cdot x \quad 
  \text{subject to } x \in \{0,1\}^3 \enspace .
  \]
  The eight points in $\{0,1\}^3$ form the feasible set in the decision space $\ZZ^3$.
  The outcome space is the set $Z=\{(-3,2), (-2,3), (-1,4), (0,0), (1,1), (2,2) \}$. 
  The two points $g=(0,0)$ and $h=(-3,2)$ are the only nondominated points of $Z$; cf.~Example~\ref{ex:complementary-cones}.
  The situation is depicted in Figure~\ref{fig:multi}.
\end{example}

Let $N$ be the set of nondominated points of $Z$. Then the set $Z + \RR^d_{\geq 0}$ agrees with $N + \RR^d_{\geq 0}$; cf.~Lemma~\ref{lem:ehrgott}.
Therefore, by Lemma~\ref{lem:monomial-maxcone}, the set $N + \RR^d_{\geq 0}$ agrees with the intersection of the real part of the \tma-tropical cone $\monomial{N}$ with the hyperplane $z_0 = 0$.
Note that here and below we identify $\RR^d$ with the real hyperplane $\smallSetOf{z\in\RR^{d+1}}{z_0=0}$ in $\TTmax^{d+1}$.
This is the reason why we started the labeling of the coordinates with zero in the previous section.
Our motivation to study monomial tropical cones comes from the following, which is a direct consequence of Theorem~\ref{thm:complementary-cones} together with Lemma~\ref{lem:interior-union-sectors}.
It recovers the decomposition in \cite[Proposition 2.3]{DaechertEtAl}.
%The set of all nondominated points of the outcome space $Z=f(X)$ is the \emph{nondominated set}.
\enlargethispage{2\baselineskip}

\begin{corollary} \label{cor:nondominated-set-monomial-cone}
  The nondominated set $N$ of $Z$ agrees with the set of those extremal generators of $\closedmonomial{N}$ which have finite coordinates.
  Moreover, if $A$ is the set of extremal generators of the complementary tropical cone $\closedcomplementarymonomial{N}$, then the set
  \[
  \bigcup_{a\in A} a - \RR_{> 0}^d
  \]
  agrees with the complement of $\monomial{N}$ in $\RR^d$.
\end{corollary}

Here, we discard the $0$th coordinate of the elements of $A$, which we can assume to be $0$ by Lemma~\ref{lem:monomial-halfspace}. 
Furthermore, we slightly abuse notation by identifying $a - \RR_{> 0}^d$ with
\[
\SetOf{x\in\RR^{d}}{0 > \max(x_j-a_j \mid j\in \supp(a))} \enspace .
\]

The next steps follow D\"achert and Klamroth \cite[\S 4 \& \S 5]{Daechert2015}.
Since our algorithm for computing the nondominated set $N$ is iterative we consider the situation where a subset $N'\subset N$ of the nondominated points is already given.
In the following let $A'\subset\TTmin^{d+1}$ be the set of extremal generators of the complementary tropical cone $\closedcomplementarymonomial{N'}$. 
The set $A'$ is never empty, even if $N'$ is.
For a given point $a\in A'$ and $i\in[d]$ we consider the auxiliary optimization problem
\begin{equation}\label{eq:scalar:step1}
  \begin{array}{ll}
    \min & z_i \\
    \mbox{subject to } & z_{j} < a_{j} \qquad \mbox{ for all } j \in \supp(a) \setminus \{0,i\} \\ 
    & z\in Z 
  \end{array}
\end{equation}
with respect to the scalar objective function $z_i$.
The scalarization technique to obtain the optimization problem \eqref{eq:scalar:step1} is known as the `$\epsilon$-constraint method' \cite[\S4.1]{Ehrgott:2005}.
If it does not have a feasible solution then there is no nondominated point contained in the set $Z \cap (a - \RR_{>0}^d)$.

Otherwise there is an optimal feasible point $w\in\RR^d$.
Then we consider as a second auxiliary optimization problem
\begin{equation}\label{eq:scalar:step2}
  \begin{array}{ll}
    \min & \sum_{j = 1}^{d} z_{j}\\
    \mbox{subject to } & z_{k} \leq w_{k} \qquad \mbox{ for all } k \in [d] \\
    & z\in Z  \enspace .
  \end{array}
\end{equation}
Notice that $w$ is a feasible solution for \eqref{eq:scalar:step2}.
To assert the existence of a finite optimal solution in \eqref{eq:scalar:step2} we assume from now on that our problem is discrete. Then the ideal point puts a lower bound on the feasible set of \eqref{eq:scalar:step2}.
The optimal solution is a new nondominated point in the complement~$N\setminus N'$.
The optimization problem \eqref{eq:scalar:step2} is a version of the `hybrid method' \cite[\S4.2]{Ehrgott:2005}. We chose this scalarization to give a clear and self-contained picture but also other scalarization methods can be applied here, cf. \cite{Ehrgott:2005}.

By Corollary~\ref{cor:nondominated-set-monomial-cone}, for each nondominated point $g$ in $N\setminus N'$ there is an extremal generator $a$ of $\closedcomplementarymonomial{N'}$ such that $g \in Z \cap (a - \RR_{> 0}^d)$; recall that $a$ may have infinite coordinates while $g$ does not.
On the other hand, by Lemma~\ref{lem:farkas}, if \eqref{eq:scalar:step1} has no feasible solution then there is no nondominated point in $(a - \RR_{> 0}^d)$.
We denote the oracle subsuming scalarizations for finding another nondominated point by \textproc{NextNonDominated}($Z,a$).
Here, $a$ is a point in $\TTmin^{d+1}$ discarding the $0$th entry as it can be assumed to be zero.
This method can be thought of as the solution arising from Equations~\eqref{eq:scalar:step1} and~\eqref{eq:scalar:step2}.
Note that we augment the resulting point by a $0$th coordinate equal to zero to fit the input format of Algorithm~\ref{algo:newextremals}.
We switch between these two points of view without any further notice to avoid notational overhead.

Now we have all the ingredients for the generation of the set of nondominated points. 
The key idea is to develop a sequence of monomial tropical cones, as in the tropical double description method~\cite{DoubleDescription:2010}.
The procedure \textproc{NewExtremals} from Algorithm~\ref{algo:newextremals} avoids redundant generators; this reduces the number of scalarizations to the bare minimum.
The nondominated points arise as the extremal generators of the max-tropical cones.
In contrast to the double description method for general tropical cones, however, in the monomial case all extremal generators of one tropical cone survive as extremal generators of the successor.
The reason is Theorem~\ref{thm:complementary-cones} which establishes that the exterior description of $\closedmonomial{G}$ agrees with the interior description of $\closedcomplementarymonomial{G}$.
This shows how the algorithms developed in~\cite{PrzybylskiGandibleuxEhrgott:2010,KLV2015,DaechertEtAl} can be considered as (dual) tropical convex hull computations.
The concept based on successive updates of the search region is well known in multicriteria optimization; cf.\ \cite[Algorithm~1]{KLV2015}.

\begin{algorithm}[htbp]
\caption{Nondominated set} \label{algo:non-dominated-set}
  \begin{algorithmic}[1]
    \Require{Outcome space $Z \subset \RR^d$, implicitly given by the objective function and the description of the feasible set.}
    \Ensure{The set of nondominated points.}
    \State $A \gets \minunit \cup e^{(0)}$ \label{line:first}
    \State $G \gets \emptyset$ %extremals of \tmi-tropical halfspace $\overline{\RR^d \setminus \left(z+\RR^d_{\geq 0}\right)}$ % tropical vertices of partial anti-domihedron
    \State $\Omega \gets \minunit$  \label{line:end-initialization} % set of confirmed extremal generators  % (vertices of) scalarizations processed \
    \While{$A \neq \Omega$}  \label{line:loop-extremals} % {$V\setminus\Omega \neq \emptyset$}
    \State pick $a$ in $A\setminus\Omega$ \label{line:pick}
    \State $g \gets$ \Call{NextNonDominated}{$Z,a$}
    \If{$g \neq $ None} \label{line:start-case-distinction}
    \State $A \gets $ \Call{NewExtremals}{$G,A,g$} \label{line:A}
    \State $G \gets G \cup \{g\}$ \label{line:add-new-nondominated}
    \Else
    \State $\Omega \gets \Omega\cup\{a\}$ \label{line:add-new-extremal}
    \EndIf \label{line:updated-new-inequality}
    \EndWhile
    \State \Return $G$
  \end{algorithmic}
\end{algorithm}

The search regions arising as input for the procedure \textproc{NextNonDominated} in Algorithm~\ref{algo:non-dominated-set} and for the program $P$ in~\cite[Algorithm~1]{KLV2015} are the same (except maybe for the ordering).
We prove the correctness of our algorithm via standard results from tropical convexity. 
In the sequel we will denote the number of nondominated points by $n=|N|$.
That number is finite as we assumed our optimization problem to be discrete.
The $d$ points in $\minunit$ are always among the extremal generators of the monomial tropical cone $\closedcomplementarymonomial{N}$, even if $N$ is empty.
These are the \emph{trivial} extremal generators.
We let $m$ be the number of the remaining extremal generators of $\closedcomplementarymonomial{N}$ which are nontrivial.

\begin{theorem}\label{thm:main}
  Algorithm~\ref{algo:non-dominated-set} returns the set of nondominated points of $Z$ after $n+m$ iterations.
\end{theorem}
\begin{proof}
  Let $N$ be the set of nondominated points of $Z$.
  To show correctness, we will maintain the following invariant:

  \begin{itemize}
  \item $G$ is a successively increasing subset of $N$. 
    Further, $A$ is the set of extremal generators of~$\closedcomplementarymonomial{G}$. 
    Finally, each point in $\Omega\subseteq A$ is also one of the extremal generators of $\closedcomplementarymonomial{N}$.
  \end{itemize}

  After the initialization in lines~\ref{line:first} to~\ref{line:end-initialization}, the invariant is fulfilled. By Lemma~\ref{lem:farkas}, in the case distinction in lines~\ref{line:start-case-distinction} to~\ref{line:updated-new-inequality}, either $g$ is a new nondominated point in $N$, or $a$ is certified to be an extremal generator of $\closedcomplementarymonomial{N}$. Hence, the invariant is preserved.
Furthermore, Proposition~\ref{prop:newextremals} implies that \textproc{NewExtremals} from Algorithm~\ref{algo:newextremals} correctly returns the set of extremal generators of $\closedcomplementarymonomial{G\cup\{g\}}$.

Observe that in lines~\ref{line:add-new-nondominated} and~\ref{line:add-new-extremal}, only new points are added to the sets $G$ and $\Omega$, respectively. In particular, the sum of cardinalities $|G| + |\Omega|$ increases in each iteration of the loop starting in line~\ref{line:loop-extremals}.

As $\Omega = \minunit$ in the beginning, these trivial extremal generators do not contribute to the number of iterations.
\end{proof}

\begin{corollary}\label{cor:complexity}
  For fixed $d$ the maximal number of scalarizations lies in $\Theta(n^{\lfloor d/2 \rfloor})$.
\end{corollary}
\begin{proof}
  By Theorem~\ref{thm:main} the number of scalarizations is bounded by $n+m$. The claim now follows from Corollary~\ref{cor:asymptotic-maximum}.
\end{proof}

Allamigeon, Gaubert and Katz observed that determining the extremal generators can be interpreted in terms of hypergraph transversals~\cite[\S 3]{AllamigeonGaubertKatz:2011}; see also \cite{BEGKM:2002}.
Depending on how the feasible set is given, this may yield more efficient methods for computing the nondominated set.

The deduction of a practicable algorithm for more general multicriteria optimization problems from their generation method is left for future work.

\begin{example}\label{ex:full-exmp}
  As one non-trivial example we examine an instance of a classical type of multicriteria optimization problem considered, e.g., in \cite{ZitzlerThiele:1999}.
  Consider 
  \[
  P =
  \begin{pmatrix}
    -1 & -4 & -3 & -1 \\
    -4 & -1 & -2 & -2 \\
    -4 & -1 & -2 & -3
  \end{pmatrix} \,,
  \quad
  W =
  \begin{pmatrix}
    2 & 1 & 1 & 1 \\
    0 & 3 & 1 & 0 \\
    0 & 1 & 1 & 2
  \end{pmatrix} \,,
  \quad
  c = 
  \begin{pmatrix}
    2 \\ 3 \\ 2
  \end{pmatrix}
  \enspace .
  \]
  Then
  \begin{equation}\label{eq:full-exmp}
    \begin{array}{l}
      \min\ P \cdot x\\
      \mbox{subject to}\quad W \cdot x \leq c\ \mbox{ with }\ x \in \{0,1\}^4
    \end{array}
  \end{equation}
  is a multidimensional 0/1-knapsack problem with three linear objective functions given by the rows of~$P$.
  Usually knapsack problems are written as maximization problems, but since Algorithm~\ref{algo:non-dominated-set} is about minimizing, the entries of the matrix $P$ are negative numbers.
  We will not distinguish between row and column vectors in the sequel.
  It is not difficult to see that the feasible points in the decision space are precisely
  \[
  (0,0,0,0)\,,\ (1,0,0,0)\,,\ (0,1,0,0)\,,\ (0,0,1,0)\,,\ (0,0,0,1) \enspace .
  \]
  Their images in the outcome space are given by
  \[
  (0,0,0)\,,\ (-1,-4,-4)\,,\ (-4,-1,-1)\,,\ (-3,-2,-2)\,,\ (-1,-2,-3) \enspace .
  \]
  We now switch to minimizing $4\cdot\1 + P\cdot x$ instead of $P\cdot x$ as in \eqref{eq:full-exmp}.
  This translation in the outcome space does not change the structure of the problem in any way, but it helps to improve the readability since we can skip many minus signs.
  The translated points in outcome space are
  \[
  (4,4,4)\,,\ (3,0,0)\,,\ (0,3,3)\,,\ (1,2,2)\,,\ (3,2,1) \enspace .
  \]

  We now demonstrate how Algorithm~\ref{algo:non-dominated-set} computes the nondominated points.
  For the scalarizations implicitly solved to obtain \textproc{NextNonDominated}, we do not make a particular choice but just consider the function as an oracle.
  This does not simplify the problem but allows us to emphasize on the crucial aspects.
  
  The initialization yields $A = \{(\infty,0,\infty,\infty),(\infty,\infty,0,\infty),(\infty,\infty,\infty,0),(0,\infty,\infty,\infty)\}$, $G = \emptyset$ and $\Omega = \{(\infty,0,\infty,\infty),(\infty,\infty,0,\infty),(\infty,\infty,\infty,0)\}$.
  Since $A\setminus\Omega$ is a singleton the only choice in line~\ref{line:pick} in the first iteration is $a = (0,\infty,\infty,\infty)$.
  By minimizing $z_3$ in the scalarization procedure \eqref{eq:scalar:step1} we obtain $(3,0,0)$ as the optimal solution.

  Hence in line~\ref{line:A} we step into Algorithm~\ref{algo:newextremals} with $h = (3,0,0)$.
  The set $A^{\geq}$ comprises $\minunit$ and its complement equals $\{(0,\infty,\infty,\infty)\}$. 
  The additionally generated points are 
%$\lambda = -\min_{i \in [3]}(0-3,\infty-0,\infty-0) = 3$.
%$\lambda = -\min_{i \in [3]}(\infty-3,0-0,\infty-0) = 0$.
%$\lambda = -\min_{i \in [3]}(\infty-3,\infty-0,0-0) = 0$.
  \begin{align*}
    (0,3,\infty,\infty) \ &= \ \min\bigl((3,3,3,3) + (\infty,0,\infty,\infty),\, (0,\infty,\infty,\infty)\bigr) \\
    (0,\infty,0,\infty) \ &= \ \min\bigl((0,0,0,0) + (\infty,\infty,0,\infty),\, (0,\infty,\infty,\infty)\bigr) \\
    (0,\infty,\infty,0) \ &= \ \min\bigl((0,0,0,0) + (\infty,\infty,\infty,0),\, (0,\infty,\infty,\infty)\bigr) \enspace ,
  \end{align*}
  and all of them are extremal; to see this check with Lemma~\ref{lem:inner-extremal}.

  \smallskip

  We arrive at the second iteration.
  Suppose we pick $a = (0,3,\infty,\infty)$, and the scalarization \eqref{eq:scalar:step1} with $i=1$ provides us with the next nondominated point $g = (0,3,3)$ which, indeed, satisfies $(0,3,3) < (3,\infty,\infty)$.

  Again we enter \textproc{NewExtremals}, now with $G = \{(3,0,0)\}$ and $h = (0,3,3)$.
  We obtain 
  $A^{\geq} = \minunit \cup \{(0,\infty,0,\infty), (0,\infty,\infty,0)\}$ and $A\setminus A^{\geq}=\{(0,3,\infty,\infty)\}$.
  Among the possible new generators derived from the pairs, e.g., we get
  \begin{align*}
    \min\bigl(3\cdot\1 + (\infty,\infty,0,\infty),\,(0,3,\infty,\infty)\bigr)\ &=\ \min\bigl(3\cdot\1 + (0,\infty,0,\infty),\,(0,3,\infty,\infty)\bigr)\\ &=\ (0,3,3,\infty)
  \end{align*}
  for $(\infty,\infty,0,\infty)$ and $(0,\infty,0,\infty)$ in $A^{\geq}$.
  The computation of the candidates for all pairs in line~\ref{line:pairs-in-out} ultimately results in the three extremal generators 
  \[
  (0,3,3,\infty)\,,\ (0,3,\infty,3)\,,\ (0,0,\infty,\infty)
  \]
  of $\closedcomplementarymonomial{G\cup\{h\}}$ which are not contained in $A^{\geq}$.
  
  \smallskip

  Now, suppose that in the next three iterations $a$ successively attains the values $(0,\infty,0,\infty)$, $(0,\infty,\infty,0)$, $(0,0,\infty,\infty)$.
  None of the corresponding scalarizations \eqref{eq:scalar:step1} has a solution, and so these points are added to $\Omega$.

  \smallskip

  In the sixth iteration for either $a = (0,3,3,\infty)$ or $a = (0,3,\infty,3)$ the next nondominated point is $g = (1,2,2)$.
  In this case we once more enter the procedure \textproc{NewExtremals}.
  There we get $A^{\geq} = \Omega$, which currently contains six points, and $A \setminus A^{\geq} = \{(0,3,3,\infty),(0,3,\infty,3)\}$.
  For instance, this yields the candidate point
  \[
  \min\bigl(2\cdot\1 + (0,\infty,0,\infty),\, (0,3,\infty,3)\bigr) \ = \ (0,3,2,3) \enspace .
  \]
  However, using Lemma~\ref{lem:check-extremal} in Algorithm~\ref{algo:newextremals} reveals that it is not extremal:
  Indeed, the minima in \eqref{eq:extremal-index} for the apices in
  \[
  G \cup \{h\} \ = \ \{(3,0,0),\,(0,3,3),\,(1,2,2)\}
  \]
  are attained at the index sets $\{0,1\}$, $\{0,2\}$ and $\{0,2\}$, respectively, but never at the index~$3$.
  Finally, the additional extremal generators are
  \[
  (0,1,3,\infty)\,,\ (0,1,\infty,3)\,,\ (0,3,2,\infty)\,,\ (0,3,\infty,2) \enspace .
  \]

  The above four extremal generators lead to four more iterations.
  In each case the corresponding scalarization is infeasible, which certifies that we already found all nondominated points.
  Hence, the Algorithm~\ref{algo:non-dominated-set} terminates and returns the set $\{(3,0,0),(0,3,3),(1,2,2)\}$.
  The total number of calls to the procedure \textproc{NextNonDominated} equals ten.
  This is also the sum of the number of nondominated points and of the extremal generators, as dictated by Theorem~\ref{thm:main}.
\end{example}

% In the special case $d\leq 3$ our result recovers the linear bound obtained by D\"achert and Klamroth~\cite{Daechert2015}.

\section{Concluding remarks and open questions}\label{sec:concluding}
\noindent
For any field $K$ consider the polynomial ring $R=K[x_1,\dots,x_d]$ in $d$ indeterminates.
An ideal $I$ in $R$ is \emph{monomial} if it is generated by monomials, i.e., products of the indeterminates.
Via identifying the monomial $x_1^{a_1}x_2^{a_2}\cdots x_d^{a_d}$ with the lattice point $(a_1,a_2,\dots,a_d)$ in the positive orthant $\RR_{\geq 0}^d$ the set $M$ of all monomials in a given monomial ideal $I$ becomes a subset of $\NN^d$.
As $I$ is an ideal it follows that $M+\NN^d\subset M$.
Dickson's Lemma says that $M$ contains a unique finite subset which minimally generates $I$.
The minimal generators of the monomial $I$ correspond to the finite extremal generators of the monomial \tma-tropical cone $\closedmonomial{M}$.
In this sense the monomial tropical cones in $\TTmax^{d+1}$ generalize the monomial ideals in $R$.
The generators of the complementary monomial tropical cone $\closedcomplementarymonomial{M}$ correspond to the irreducible components of $I$.
That is to say, Theorem~\ref{thm:complementary-cones} generalizes the Alexander duality of monomial ideals \cite[\S5.2]{MillerSturmfels:2005}.
From this one can also see how the tropical convex hull computation generalizes the irreducible component computation for monomial ideals.
In the special case where the generators are squarefree, i.e., their exponent vectors consist of zeros and ones, the Alexander duality of monomial ideals agrees with the Alexander duality of finite simplicial complexes. 

For $d=3$ the common intersection of $\monomial{M}$ with $\complementarymonomial{M}$ is known as the \emph{staircase surface} of $I$; cf.~\cite[Chap.~3]{MillerSturmfels:2005}.
We denote its generalization to arbitrary $d$ as $\Sigma(I)$.
This is precisely the topological boundary of the projection of a monomial tropical cone in $\RR^{d+1}$ to $\RR^{d+1}/\RR\1$.
The covector decomposition of a tropical cone studied in \cite{JoswigLoho:2016} induces a polyhedral subdivision of $\Sigma(I)$, and this agrees with the `hull complex'; cf.~\cite[\S4.5]{MillerSturmfels:2005}.
The following seems promising.
\begin{question}
  Give an interpretation of the planar resolution algorithm from \cite[\S3.5]{MillerSturmfels:2005} and the hull resolution from \cite[\S4.4]{MillerSturmfels:2005} in terms of tropical convexity.
\end{question}

\smallskip

In view of the tropical upper bound theorem (Theorem~\ref{thm:upper-bound}) the number $m$ of extremal generators of a (monomial) tropical cone given as the intersection of $n$ tropical halfspaces is bounded by $U(d+n,d)$; cf.~\eqref{eq:upper-bound}.
Equivalently, the number $m$ of extremal generators translates to the number of scalarizations required for a $d$-criteria optimization problem with $n$ nondominated points.
It is known that that bound is not tight for all parameters; cf.~\cite{AGK:ExtremePoints2011}.
\begin{question}\label{prob:bound}
  Determine the exact upper bound for $m$ as a function of $n$ and $d$.
\end{question}
Already Bayer, Peeva and Sturmfels derive an upper bound theorem in \cite[Cor. 6.2 \& Thm. 6.3]{BayerPeevaSturmfels:1998} for monomial ideals from the upper bound theorem for polytopes.
It is known from work of Ho\c{s}ten and Morris \cite{HostenMorris:1999} that that upper bound can only be attained for special parameters; cf.~\cite[Thm.~6.33]{MillerSturmfels:2005}.

\smallskip

Related to Question~\ref{prob:bound} is the question what `combinatorial types' of monomial tropical cones can occur.
In contrast to general tropical cones each monomial tropical cone has a unique minimal exterior description in terms of the extremal generators of its complementary monomial tropical cone.
This leads to a well-defined notion of \emph{vertex-facet incidences} for monomial tropical cones.
\begin{question}
  Which bipartite graphs occur as the vertex-facet incidence graphs of monomial tropical cones?
\end{question}
In \cite{DaechertEtAl}, the ``neighborhood relation'' of the facets is applied to devise a more combinatorial update procedure for computing all nondominated points.  Studying the vertex-facet incidences further might unveil new aspects of their algorithm.

\smallskip

Finally, it is a natural question to ask how far our approach can be generalized.
\begin{question}
  To what extent does our approach generalize to multicriteria optimization problems which are not discrete?
\end{question}
It seems plausible to explore more general semigroup rings; e.g., cf.~\cite[Chap.~7]{MillerSturmfels:2005}.

\section*{Acknowledgments}
\noindent
We are much indebted to Ben Burton who brought the subject of multicriteria optimization and especially the work of D\"achert and Klamroth \cite{Daechert2015} to our attention.
Further, we would like to thank Kathrin Klamroth for pointing out \cite{KLV2015} and the work by Kaplan et al. \cite{KaplanEtAl:2008}.
We are grateful to the reviewer for helpful comments and for bringing \cite{PrzybylskiGandibleuxEhrgott:2010} to our attention.

\bibliographystyle{amsplain}
\bibliography{main}

\providecommand{\bysame}{\leavevmode\hbox to3em{\hrulefill}\thinspace}
\providecommand{\MR}{\relax\ifhmode\unskip\space\fi MR }
\providecommand{\MRhref}[2]{%
  \href{http://www.ams.org/mathscinet-getitem?mr=#1}{#2}
}
\providecommand{\href}[2]{#2}
\begin{thebibliography}{10}

\bibitem{ABGJ:2015}
Xavier Allamigeon, Pascal Benchimol, St\'ephane Gaubert, and Michael Joswig,
  \emph{Tropicalizing the simplex algorithm}, SIAM J. Discrete Math.
  \textbf{29} (2015), no.~2, 751--795. \MR{3336300}

\bibitem{DoubleDescription:2010}
Xavier Allamigeon, St\'ephane Gaubert, and \'Eric Goubault, \emph{The tropical
  double description method}, S{TACS} 2010: 27th {I}nternational {S}ymposium on
  {T}heoretical {A}spects of {C}omputer {S}cience, LIPIcs. Leibniz Int. Proc.
  Inform., vol.~5, Schloss Dagstuhl. Leibniz-Zent. Inform., Wadern, 2010,
  pp.~47--58. \MR{2853909}

\bibitem{AGK:ExtremePoints2011}
Xavier Allamigeon, St\'ephane Gaubert, and Ricardo~D. Katz, \emph{The number of
  extreme points of tropical polyhedra}, J. Combin. Theory Ser. A \textbf{118}
  (2011), no.~1, 162--189. \MR{2737191}

\bibitem{AllamigeonGaubertKatz:2011}
\bysame, \emph{Tropical polar cones, hypergraph transversals, and mean payoff
  games}, Linear Algebra Appl. \textbf{435} (2011), no.~7, 1549--1574.
  \MR{2810655}

\bibitem{BasuPollackRoy:2006}
Saugata Basu, Richard Pollack, and Marie-Fran\c{c}oise Roy, \emph{Algorithms in
  real algebraic geometry}, second ed., Algorithms and Computation in
  Mathematics, vol.~10, Springer-Verlag, Berlin, 2006. \MR{2248869}

\bibitem{BayerPeevaSturmfels:1998}
Dave Bayer, Irena Peeva, and Bernd Sturmfels, \emph{Monomial resolutions},
  Math. Res. Lett. \textbf{5} (1998), no.~1-2, 31--46. \MR{1618363}

\bibitem{BEGKM:2002}
Endre Boros, Khaled~M. Elbassioni, Vladimir~A. Gurvich, Leonid~G. Khachiyan,
  and Kazuhisa Makino, \emph{Dual-bounded generating problems: all minimal
  integer solutions for a monotone system of linear inequalities}, SIAM J.
  Comput. \textbf{31} (2002), no.~5, 1624--1643. \MR{1936663}

\bibitem{Butkovic10}
Peter Butkovi{\v{c}}, \emph{Max-linear systems: theory and algorithms},
  Springer Monographs in Mathematics, Springer-Verlag London, Ltd., London,
  2010. \MR{2681232 (2011e:15049)}

\bibitem{Daechert2015}
Kerstin D{\"a}chert and Kathrin Klamroth, \emph{A linear bound on the number of
  scalarizations needed to solve discrete tricriteria optimization problems},
  Journal of Global Optimization \textbf{61} (2015), no.~4, 643--676.

\bibitem{DaechertEtAl}
Kerstin D\"achert, Kathrin Klamroth, Renaud Lacour, and Daniel Vanderpooten,
  \emph{Efficient computation of the search region in multi-objective
  optimization}, European J. Oper. Res. \textbf{260} (2017), no.~3, 841--855.
  \MR{3626174}

\bibitem{DeLoeraHemmeckeKoeppe:2013}
Jes\'us~A. De~Loera, Raymond Hemmecke, and Matthias K\"oppe, \emph{Algebraic
  and geometric ideas in the theory of discrete optimization}, MOS-SIAM Series
  on Optimization, vol.~14, Society for Industrial and Applied Mathematics
  (SIAM), Philadelphia, PA; Mathematical Optimization Society, Philadelphia,
  PA, 2013. \MR{3024570}

\bibitem{DevelinYu:2007}
Mike Develin and Josephine Yu, \emph{Tropical polytopes and cellular
  resolutions}, Experiment. Math. \textbf{16} (2007), no.~3, 277--291.
  \MR{2367318}

\bibitem{Ehrgott:2005}
Matthias Ehrgott, \emph{Multicriteria optimization}, second ed.,
  Springer-Verlag, Berlin, 2005. \MR{2143243}

\bibitem{FukudaProdon:1996}
Komei Fukuda and Alain Prodon, \emph{Double description method revisited},
  Combinatorics and computer science ({B}rest, 1995), Lecture Notes in Comput.
  Sci., vol. 1120, Springer, Berlin, 1996, pp.~91--111. \MR{1448924}

\bibitem{Gaubert:PhD}
St\'ephane Gaubert, \emph{Theory of linear systems over dioids}, Ph.D. thesis,
  Ecole Nationale Sup\'erieure des Mines de Paris, 1992.

\bibitem{GaubertKatz07}
St{\'e}phane Gaubert and Ricardo~D. Katz, \emph{The {M}inkowski theorem for
  max-plus convex sets}, Linear Algebra Appl. \textbf{421} (2007), no.~2-3,
  356--369. \MR{2294348 (2007k:52001)}

\bibitem{HerzogHibi:2011}
J\"urgen Herzog and Takayuki Hibi, \emph{Monomial ideals}, Graduate Texts in
  Mathematics, vol. 260, Springer-Verlag London, Ltd., London, 2011.
  \MR{2724673}

\bibitem{HostenMorris:1999}
Serkan Ho\c{s}ten and Walter~D. Morris, Jr., \emph{The order dimension of the
  complete graph}, Discrete Math. \textbf{201} (1999), no.~1-3, 133--139.
  \MR{1687882}

\bibitem{Tropical+halfspaces}
Michael Joswig, \emph{Tropical halfspaces}, Combinatorial and computational
  geometry, Math. Sci. Res. Inst. Publ., vol.~52, Cambridge Univ. Press,
  Cambridge, 2005, pp.~409--431. \MR{2178330}

\bibitem{JoswigLoho:2016}
Michael Joswig and Georg Loho, \emph{Weighted digraphs and tropical cones},
  Linear Algebra Appl. \textbf{501} (2016), 304--343. \MR{3485070}

\bibitem{KaplanEtAl:2008}
Haim Kaplan, Natan Rubin, Micha Sharir, and Elad Verbin, \emph{Efficient
  colored orthogonal range counting}, SIAM J. Comput. \textbf{38} (2008),
  no.~3, 982--1011. \MR{2421075}

\bibitem{KLV2015}
Kathrin Klamroth, Renaud Lacour, and Daniel Vanderpooten, \emph{On the
  representation of the search region in multi-objective optimization},
  European J. Oper. Res. \textbf{245} (2015), no.~3, 767--778. \MR{3345908}

\bibitem{Tropical+Book}
Diane Maclagan and Bernd Sturmfels, \emph{Introduction to tropical geometry},
  Graduate Studies in Mathematics, vol. 161, American Mathematical Society,
  Providence, RI, 2015. \MR{3287221}

\bibitem{Markwig:2010}
Thomas Markwig, \emph{A field of generalised {P}uiseux series for tropical
  geometry}, 2010, pp.~79--92. \MR{2759691 (2012e:14126)}

\bibitem{McMullen:1970}
Peter McMullen, \emph{The maximum numbers of faces of a convex polytope},
  Mathematika \textbf{17} (1970), 179--184. \MR{0283691}

\bibitem{MillerSturmfels:2005}
Ezra Miller and Bernd Sturmfels, \emph{Combinatorial commutative algebra},
  Graduate Texts in Mathematics, vol. 227, Springer-Verlag, New York, 2005.
  \MR{2110098}

\bibitem{PrzybylskiGandibleuxEhrgott:2010}
Anthony Przybylski, Xavier Gandibleux, and Matthias Ehrgott, \emph{A two phase
  method for multi-objective integer programming and its application to the
  assignment problem with three objectives}, Discrete Optim. \textbf{7} (2010),
  no.~3, 149--165. \MR{2651544}

\bibitem{Ziegler:Lectures+on+polytopes}
G{\"u}nter~M. Ziegler, \emph{Lectures on polytopes}, Graduate Texts in
  Mathematics, vol. 152, Springer-Verlag, New York, 1995. \MR{MR1311028
  (96a:52011)}

\bibitem{ZitzlerThiele:1999}
Eckart Zitzler and Lothar Thiele, \emph{Multiobjective evolutionary algorithms:
  A comparative case study and the strength pareto approach}, IEEE Transactions
  on Evolutionary Computation \textbf{3} (1999), 257--271.

\end{thebibliography}

\end{document}